\documentclass{amsart}
\usepackage{graphicx}
\usepackage{mathrsfs}
\usepackage{amssymb}
\usepackage{color}
\usepackage{amsmath,bm}
\usepackage[backref]{hyperref}
\usepackage{enumerate}
\newtheorem{thm}{Theorem}[section]

\newtheorem{lem}[thm]{Lemma}

\newtheorem{conj}[thm]{Conjecture}
\theoremstyle{definition}
\newtheorem{defn}[thm]{Definition}
\newtheorem{rem}[thm]{Remark}
\newtheorem{exa}[thm]{Example}
\numberwithin{equation}{section}

\begin{document}

\title[]{Upper bound of discrepancies of divisors computing minimal log discrepancies on surfaces}%
\author{Bingyi Chen}%
\address{Bingyi Chen, Yau Mathematical Sciences Center,
Tsinghua University,
Beijing, 100084, P. R. China.}
\email{bychen@mail.tsinghua.edu.cn}

\thanks{}%

\begin{abstract}
Fix a subset $I\subseteq \mathbb R_{>0}$ such that 
$$\gamma:=\inf\{ \sum_{i}n_ib_i-1>0 \mid n_i\in \mathbb Z_{\geq 0}, b_i\in I \}>0.$$ We give an explicit upper bound $\ell(\gamma)\in O(1/\gamma^2)$ as $\gamma\to 0$, such that for any smooth surface $A$ of arbitrary characteristic with a closed point 0 and an $\mathbb R$-ideal $\mathfrak{a}$ with exponents in $I$, there always exists a prime divisor $E$ over $A$ computing the minimal log discrepancy of $(A,\mathfrak{a})$ at 0 and with its log discrepancy  $k_E+1\leq \ell(\gamma)$. Some examples indicate that our bound is optimal.

\end{abstract}
\maketitle
\section{Introduction}
Let $A$ be a smooth variety over an algebraically closed field $k$ and $0\in A$  a closed point. Let $\mathfrak{a}$ be an $\mathbb{R}$-ideal on $A$, that is, a formal product $\mathfrak{a}=\prod_{j=1}^r \mathfrak{a}_j^{\lambda_j}$ where each $\lambda_j$ is a positive real number and each $\mathfrak{a}_j$ is a non-zero coherent ideal sheaf on $A$. Denote by $\rm{mld}(0;A,\mathfrak{a})$ the minimal log discrepancy (mld, for short) of the pair $(A,\mathfrak{a})$ at 0 and denote by $a(E;A,\mathfrak{a})$ the log discrepancy of $E$ with repect to $(A,\mathfrak{a})$. We say a prime divisor $E$ with the center 0 computes $\rm{mld}(0;A,\mathfrak{a})$ if $a(E;A,\mathfrak{a})$ equals to $\rm{mld}(0;A,\mathfrak{a})$ or is negative. Musta\c t\v a and Nakamura \cite{MN} proposed a conjecture, says Musta\c t\v a-Nakamura conjecture (MN conjecture, for short), on the boundness of the discrepancies of divisors computing mld on a fixed klt germ. Although the original statement is more general, we state the conjecture only for smooth varieties since we will focus on smooth surfaces in this paper.

\begin{conj}[MN conjecture for smooth varieties]\label{conj}
Let $A$ be a smooth variety of dimension $N$ over an algebraically closed field with a closed point $0$. Given a finite subset $I$ of the positive real numbers, there exists a positive integer $\ell_{N,I}$ depending only on $N$ and $I$ such that for any $\mathbb{R}$-ideal $\mathfrak{a}$  with exponents in $I$, there exists a prime divisor $E$ over $A$ that computes {\rm{mld}}$(0;A,\mathfrak{a})$ and such that its log discrepancy $k_E+1\leq \ell_{N,I}$. 
\end{conj}
MN conjecture is important in birational geometry.  It was proved in \cite{MN} that this conjecture implies the ACC conjecture for mld on a fixed germ in characteristic 0. Kawakita \cite[Theorem 4.6]{Ka} proved that the converse also holds for threefolds. Besides, MN conjecture also plays an important role on basic properties of singularities, for example, it guarantees lower semi-continuity of Mather-Jacobian mld and also stability of  Mather-Jacobian log canonicity under small deformations, which are not known in positive characteristic (see Theorem 1.3 and Proposition 1.7 in \cite{Ish1}).

For surface germs in characteristic 0, MN conjecture was proved by Musta\c t\v a and Nakamura \cite{MN}, and Alexeev \cite[Lemma 3.7]{Ale} proved that it still holds when $I$ is just a DCC set but not a finite set, under the assumption that  $\mathfrak{a}$ is locally principle (i.e. an $\mathbb R$-divisor) and  \rm{mld}$(0;A,\mathfrak{a})\geq 0$. One can see \cite[Theorem B.1]{CH} for a proof of Alexeev's result. Han and Luo \cite[Theorem 1.3]{HL} extended MN conjecture to the case that the germ is not necessary fixed and proved it for surfaces in characteristic 0. As one of key steps in their proof, they showed MN conjecture for the smooth surface germ in characteristic 0 in a more general setting: $I$ is a subset of $\mathbb R_{>0}$ such that
$$\gamma:=\inf\{ \sum_{i}n_ib_i-1>0 \mid n_i\in \mathbb Z_{\geq 0}, b_i\in I \}>0.$$
Note that this condition is satisfied for any DCC sets (see \cite[Lemma 3.2]{HL}). They also gived an explicit upper bound which only depends on $\gamma$.

\begin{thm}\cite[Theorem 1.3]{HL}
Given a subset $I$ of the positive real numbers such that 
$$\{ \sum_{i}n_ib_i-1>0 \mid n_i\in \mathbb Z_{\geq 0}, b_i\in I \}\subseteq [\gamma,+\infty)$$ 
for some $\gamma\in(0,1]$. Let $X$ be a smooth surface over $\mathbb C$ with a closed point 0 and $B$ an effective $\mathbb R$-divisors on $X$ with coefficients in $I$ such that {\rm{mld}}$(0;X,B)\geq 0$. Then there exists a prime divisor $E$ over $X$ that computes {\rm{mld}}$(0;X,B)$ and with its log discrepancy $k_E+1\leq 2^{N_0}$, where 
$$N_0=\left\lfloor 1+ \frac{32}{\gamma^2}+\frac{1}{\gamma}\right\rfloor.$$
\end{thm}
The upper bound they gived grows roughly like $2^{1/\gamma^2}$ when $\gamma$ tends to 0. In this paper, we will use a completely different approach to give a smaller bound that belongs to $O(\frac{1}{\gamma^2})$ as $\gamma\to 0$, which works in arbitrary characteristic.

The idea comes from Ishii \cite{Ish2}. In the paper, Ishii proved that MN conjecture holds for any smooth surface $A$ in arbitrary characteristic and she pointed out that the upper bound in the conjecture can be calculated by using toric geometry for smooth surfaces. Indeed, she proved that for every $\mathbb{R}$-ideal $\mathfrak{a}$ on a smooth surface $A$ there is a monomial $\mathbb{R}$-ideal $\mathfrak{a}_*$ on $\mathbb{A}_k^2$ with same exponents as $\mathfrak{a}$, such that $\rm{mld}(0;A,\mathfrak{a})=\rm{mld}(0;\mathbb{A}_k^2,\mathfrak{a})$ and $a(E;A,\mathfrak{a})\leq a(E;\mathbb{A}_k^2,\mathfrak{a}_*)$ for any prime divisor $E$ with the center 0 (here we identify prime divisors over $A$ with the center 0 and those over $\mathbb A_k^2$ with the center 0). 
Thus every prime divisor computing $\rm{mld}(0;\mathbb{A}_k^2,\mathfrak{a}_*)$ also computes $\rm{mld}(0;A,\mathfrak{a})$. Then the problem is reduced to the one on the pairs of monomial $\mathbb{R}$-ideals on $\mathbb{A}_k^2$ and can be solved by combinatorics.

The following is the main theorem of this paper.
\begin{thm}\label{thm1.3}
Let $A$ be a smooth surface over an algebraically closed field of arbitrary characteristic and let $0$ be a closed point on $A$. Given a subset $I$ of the positive real numbers, denote $e:=\inf I$ and 
$$\gamma:=\inf\left\{ \sum_{i}n_ib_i-1>0 \mid n_i\in \mathbb Z_{\geq 0}, b_i\in I \right\} .$$
Suppose $\gamma>0$, then 

\leftmargini=6mm
\begin{itemize}
\item[{\rm(1)}] for any $\mathbb{R}$-ideal $\mathfrak{a}$  with exponents in $I$ such that {\rm{mld}}$(0;A,\mathfrak{a})\geq 0$, there exists a prime divisor $E$ over $A$ that computes {\rm{mld}}$(0;A,\mathfrak{a})$ and such that its log discrepancy 
$$k_E+1\leq \max\left\{\left\lfloor \frac{\gamma+1}{e\gamma}\right\rfloor+\left\lceil \frac{\gamma+1}{e}\right\rceil,2\right\};$$

\item[{\rm(2)}] for any $\mathbb{R}$-ideal $\mathfrak{a}$  with exponents in $I$ such that {\rm{mld}}$(0;A,\mathfrak{a})=-\infty$, there exists a prime divisor $E$ over $A$ that computes {\rm{mld}}$(0;A,\mathfrak{a})$ and such that its log discrepancy 
$$k_E+1\leq \left\lfloor \frac{\gamma+1}{e\gamma}\right\rfloor+\left\lceil \frac{\gamma+1}{e}\right\rceil+1.$$
\end{itemize}
\end{thm}

\begin{rem}
Note that we always have $e\geq \gamma$. Hence the bound belongs to $O(\frac{1}{\gamma^2})$ as $\gamma\to 0$.
\end{rem}
\begin{rem}
Over a smooth surface, every exceptional divisor $E$ can be obtained by a finite sequence of blowing-ups of points and its discrepancy $k_E$ is equal to the number of necessary blowing-ups to obtain $E$. Therefore we obtain a upper bound of the number of necessary blowing-ups of points to get a divisor computing the mld for smooth surfaces.
\end{rem}
The following two examples indicate that our bound is optimal. The proofs of the examples can be found in Section 5.
\begin{exa}\label{e1}
Fix a positive integer $n\geq 2$. Denote
$$e:=\frac{1}{n-1}+\frac{1}{n^2}.$$
Let $I=\{e\}$. Then $e=\inf I$ and 
$$\gamma:=\inf\left\{ \sum_{i}n_ib_i-1>0 \mid n_i\in \mathbb Z_{\geq 0}, b_i\in I \right\}=\frac{n-1}{n^2}.$$
By calculation, in this case the bound in Theorem \ref{thm1.3} (1) is $$n^2+n-1.$$
Let $\mathfrak a=(x^{n^2},y^{n-1})^{e}$ on $\mathbb A_k^2= \text{Spec } k[x,y]$. Then $\rm{mld}(0;\mathbb{A}_k^2,\mathfrak{a})=0$ and the toric divisor corresponding to the vector $(n-1,n^2)$ computes the mld with its log discrepancy equal to $n^2+n-1$. Moreover, any prime divisor that computes the mld satisfies that its log discrepancy $\geq n^2+n-1$. Therefore the bound is optimal.
\end{exa}

\begin{exa}\label{e2}
Fix a positive integer $n$. Let $I=\{1/n\}$. Then $e:=\inf I=1/n$ and 
$$\gamma:=\inf\left\{ \sum_{i}n_ib_i-1>0 \mid n_i\in \mathbb Z_{\geq 0}, b_i\in I \right\}=\frac{1}{n}.$$
By calculation, in this case the bound in Theorem \ref{thm1.3} (2) is $$(n+1)^2+1.$$
Let $\mathfrak a=(x^{n^2+n+1},y^{n+1})^{1/n}$ on $\mathbb A_k^2= \text{Spec } k[x,y]$. Then $\rm{mld}(0;\mathbb{A}_k^2,\mathfrak{a})=-\infty$ and the toric divisor corresponding to the vector $(n+1,n^2+n+1)$ computes the mld with its log discrepancy equal to  $(n+1)^2+1$. Moreover, any prime divisor that computes the mld satisfies that its log discrepancy $\geq (n+1)^2+1$. Therefore the bound is optimal.
\end{exa}

At the end of this section we introduce some notations that will be used in the following sections.

\medskip
\noindent\textbf{Notation.}

\begin{enumerate}[(1)]

\item For $\bm{a},\bm{b} \in \mathbb{R}^2$ ($\bm{a}\neq \bm{b}$), we denote the unique line passing through $\bm{a}$ and $\bm{b}$ by $\overline{\bm{ab}}$. 

\item Given a line $L$ not paralleling to the y-axis in $\mathbb{R}^2$, we decompose $\mathbb{R}^2$ into three parts 
$$\mathbb R^2=L^{+}\cup L\cup L^{-},$$
where
\begin{align*}
L^{+}=\{(x_0,y_0)\in \mathbb{R}^2 \mid & ~y_0>\text{the second coordinate of} \\
&\text{the intersection point of }x=x_0 \text{ and } L\};\\
L^{-}=\{(x_0,y_0)\in \mathbb{R}^2 \mid  & ~y_0<\text{the second coordinate of } \\
&\text{the intersection point of }x=x_0 \text{ and } L\}.
\end{align*}
 
\item For any $\bm{a} \in \mathbb R^2$, we denote its first coordinate by $\bm{a}_x$ and its second coordinate by $\bm{a}_y$. For any $\bm{a},\bm{b} \in \mathbb R^2$, we denote $\langle \bm{a},\bm{b}\rangle=\bm{a}_x\bm{b}_x+\bm{a}_y \bm{b}_y$.

\item We write $\bm{1}$ for the vector $(1,1)$ and $\bm{0}$ for the vector $(0,0)$.

\item We denote by $\mathbb N$ the set of all positive integers.



\item Let $\lambda$ be a positive real number. For any real number $a$, we denote
$$\lceil a \rceil_{\lambda}=\min\{n\lambda\mid n\in \mathbb Z \text{ and } n\lambda\geq a\},$$
$$\lfloor a \rfloor_{\lambda}=\max\{n\lambda\mid n\in \mathbb Z \text{ and } n\lambda\leq a\}.$$
The absence of subscripts means $\lambda=1$. It's not hard to check that
$$\frac{\lceil a \rceil_{\lambda}}{\lambda}=\left\lceil \frac{a}{\lambda} \right\rceil \quad \text{and} \quad \frac{\lfloor a \rfloor_{\lambda}}{\lambda}=\left\lfloor \frac{a}{\lambda} \right\rfloor.$$

\item Let $B=\sum b_i B_i$ be a divisor on a variety where the $B_i$ are prime divisors. Let $\epsilon$ be a real number, then we denote
$$B_{\leq \epsilon}=\sum_{b_i\leq \epsilon} b_i B_i\quad
\text{and} \quad  B_{< \epsilon}=\sum_{b_i< \epsilon} b_i B_i.$$

\end{enumerate}

 \section*{Acknowledgement}
The author expresses his sincere gratitude to Shihoko Ishii for suggesting the problem and for her constant support of this project. The author would also like to thank Jingjun Han for very helpful discussions.

\section{Preliminaries}
Let $k$ ba an algebraically closed field of arbitrary characteristic.

\begin{defn}
Let $A$ be a smooth variety over $k$ and $E$ a prime divisor over $A$, i.e. a prime divisor on a normal variety $Y$ with a birational morphism $f:Y\rightarrow A$. We may write $K_Y=f^*K_A+D$. Then the center of $E$ on $X$ is defined as the image of $E$ on $X$ under the morphism $f$, the discrepancy $k_E$ of $E$ is defined as $\text{mult}_E D$ and the log discrepancy of $E$ is defined as $k_E+1$. 
\end{defn}

\begin{defn}\label{d1}
Let $A$ be a smooth variety over $k$ and $\mathfrak{a}$ an $\mathbb{R}$-ideal on $A$, i.e. a formal product $\prod_{j=1}^r \mathfrak{a}_j^{\lambda_j}$ where each $\lambda_j$ is a positive real number and each $\mathfrak{a}_j$ is a non-zero coherent ideal sheaf on $A$. 
The support of $\mathfrak{a}$ is defined to be $\cup_{j=1}^r V(\mathfrak{a}_j)$, where $V(\mathfrak{a}_j)$ is the zero set of the ideal sheaf $\mathfrak{a}_j$. For a prime divisor $E$ over $A$, the log discrepancy of $E$ with respect to $(A,\mathfrak a)$ is defined to be 
$$a(E;A,\mathfrak a)=k_E+1-\sum_{j=1}^r \lambda_j {\rm val}_E(\mathfrak a_j).$$
The minimal log discrepancy of the pair $(A,\mathfrak{a})$ at a closed point 0 is given by
$${\rm mld}(0;A,\mathfrak{a})=\inf\{a(E;A,\mathfrak a)\mid E \text{ is a prime divisor over $A$ with the center 0}\}.$$
We say a prime divisor $E$ over $A$ with the center 0 computes ${\rm mld}(0;A,\mathfrak{a})$ if
$$a(E;A,\mathfrak a)= 
\begin{cases}
{\rm mld}(0;A,\mathfrak{a}),  & \text{if}\quad  \rm{mld}(0;A,\mathfrak{a})\geq 0, \\
< 0, & \text{if} \quad \rm{mld}(0;A,\mathfrak{a})=-\infty.
\end{cases}
$$ 
\end{defn}

\begin{defn}
An $\mathbb R$-ideal  $\mathfrak{a}=\prod_{j=1}^r \mathfrak{a}_j^{\lambda_j}$ on $\mathbb{A}_k^2=\text{Spec } k[x,y]$ is called a monomial $\mathbb R$-ideal if each $\mathfrak{a}_j$ is generated by monomials.
\end{defn}

Let $A$ be a smooth surface over $k$ with a closed point 0 and an $\mathbb{R}$-ideal $\mathfrak{a}$ on $A$. In the proof of  \cite[Theorem 1.4]{Ish2}, Ishii proved that there are a regular system of parameters $x,y$ of $\mathcal{O}_{A,0}$ and a monomial $\mathbb{R}$-ideal $\mathfrak{a}_*$ on $\mathbb{A}_k^2$ with same exponents as $\mathfrak{a}$ such that 

\begin{enumerate}[\quad (1)]
\item $\rm{mld}(0;A,\mathfrak{a})=\rm{mld}(0;\mathbb{A}_k^2,\mathfrak{a}_*)$;
\item if we identify prime divisors over $A$ with the center 0 and those over $\mathbb A_k^2$ with the center 0 in terms of the \'etale morphism from $A$ to $\mathbb A_k^2$ induced by  parameters $x,y$, then
$$a(E;A,\mathfrak{a})\leq a(E;\mathbb{A}_k^2,\mathfrak{a}_*)$$
for any prime divisor $E$ over $A$ (or over $\mathbb A_k^2$) with the center 0.
\end{enumerate}
Hence every prime divisor computing $\rm{mld}(0;\mathbb{A}_k^2,\mathfrak{a}_*)$ also computes $\rm{mld}(0;A,\mathfrak{a})$. Therefore, the problem is reduced into the one on the pairs of monomial $\mathbb{R}$-ideals on $\mathbb{A}_k^2$.

\section{Newton polytope}
Let $\mathfrak a=\prod_{j=1}^r \mathfrak{a}_j^{\lambda_j}$ be a monomial $\mathbb R$-ideal on $\mathbb A^2_k$ and write $\text{Supp } \mathfrak a$ for the set 
$$\left\{\sum_j \lambda_j(a_j,b_j) \in \mathbb R_{\geq 0}^2\mid (a_j,b_j) \text{ is the exponent of a monomial in } \mathfrak{a}_j\right\}.$$
We denote by $\Gamma(\mathfrak a)$ the convex hull of $(\text{Supp } \mathfrak a+\mathbb R_{\geq 0}^2)$ in $\mathbb R_{\geq 0}^2$, which is called the Newton polytope of $\mathfrak a$. Then $\Gamma(\mathfrak a)$ has finite vertices and every compact 1-dimensional faces has slope less than 0. Every vertex $\bm{a}$ of $\Gamma(\mathfrak a)$ can be written as the form $\sum_{j=1}^r \lambda_j \bm{a}_j$, where $\bm{a}_j\in \mathbb Z_{\geq 0}^2$ is a vertex of $\Gamma(\mathfrak a_j)$.

\begin{lem}\label{lem3.1}
Let $\mathfrak a=\prod_{j=1}^r \mathfrak{a}_j^{\lambda_j}$ be a monomial $\mathbb R$-ideal on $\mathbb{A}_k^2$. If $\bm{a},\bm{b} \in \mathbb R_{\geq 0}^2$ are two vertices of a 1-dimensional compact faces of  $\Gamma(\mathfrak a)$ such that $\bm{a}_y>\bm{b}_y$, then there exist $0<\alpha\leq 1$ and $j=1,\cdots,r$ such that $(\bm{a}-\bm{c})\in (\lambda_j\mathbb Z)^2$, where $\bm{c}=(1-\alpha) \bm{a}+\alpha \bm{b}$.
\end{lem}

\begin{proof}
We may write $\bm{a}=\sum_{j=1}^r \lambda_j \bm{a}_j$ and $\bm{b}=\sum_{j=1}^r \lambda_j \bm{b}_j$, where $\bm{a}_j, \bm{b}_j \in \mathbb Z_{\geq 0}^2$ are vertices of $\Gamma(\mathfrak a_j)$. Since $\bm{a}\neq \bm{b}$, there exists $j$ such that $\bm{a}_j\neq \bm{b}_j$. Without loss of generality, we may suppose that $\bm{a}_1\neq \bm{b}_1$. Let 
$$\bm{c}:=\lambda_1\bm{b}_1+\sum_{j=2}^r \lambda_j \bm{a}_j\quad \text{ and }\quad \bm{d}:=\lambda_1\bm{a}_1+\sum_{j=2}^r \lambda_j \bm{b}_j.$$ Then $\bm{c},\bm{d}\in \Gamma(\mathfrak a)$. Since $\bm{a},\bm{b}$ are vertices of $\Gamma(\mathfrak{a})$, both $\bm{c}$ and $\bm{d} \in \overline{\bm{ab}}\cup\overline{\bm{ab}}^+$. On the other hands, $\bm{c}+\bm{d}=\bm{a}+\bm{b}$, this implies that both $\bm{c},\bm{d} \in \overline{\bm{ab}}$. So we have 
$$\bm{c},\bm{d} \in \{(1-\alpha) \bm{a}+\alpha \bm{b}\mid 0\leq \alpha\leq 1\}$$
since $\bm{a},\bm{b}$ are vertices of $\Gamma(\mathfrak{a})$.
Note that $\bm{a}-\bm{c}=\lambda_1(\bm{a}_1-\bm{b}_1)\in (\lambda_1 \mathbb Z)^2$ and $\bm{a}-\bm{c}\neq 0$, the proof is completed.
\end{proof}

\begin{lem}\label{lem3.2}
Given a subset $I$ of positive real numbers. Let $\mathfrak a=\prod_{j=1}^r \mathfrak{a}_j^{\lambda_j}$ be a monomial $\mathbb R$-ideal on $\mathbb A^2_{k}$ with exponents in $I$. If $\bm{a}$ is a vertex of $\Gamma(\mathfrak a)$, then 
$$\bm{a}_x,\bm{a}_y\in \left\{ \sum_{i}n_ib_i\mid n_i\in \mathbb Z_{\geq 0}, b_i\in I \right\}.$$
\end{lem}
\begin{proof}
It follows from the fact that every vertex $\bm{a}$ of $\Gamma(\mathfrak a)$ can be written as the form $\sum_{j=1}^r \lambda_j \bm{a}_j$, where $\bm{a}_j\in \mathbb Z_{\geq 0}^2$ is a vertex of $\Gamma(\mathfrak a_j)$.
\end{proof}

Let $\mathfrak a$ be a monomial $\mathbb R$-ideal on $\mathbb A_k^2$ with the Newton polytope $\Gamma$. For any $\bm{p}=(p_1,p_2)\in \mathbb N^2$ such that $p_1,p_2$ are coprime, we denote by $E_{\bm{p}}$ the prime toric divisor over $\mathbb A_k^2$ which corresponds to the 1-dimensional cone 
$\bm{p}\mathbb R_{\geq 0}$, then we have $k_{E_{\bm{p}}}+1=\langle \bm{p},\bm{1}\rangle$ and 
${\rm val}_{E_{\bm{p}}}(\mathfrak a)=\langle \bm{p},\Gamma\rangle$,  where  
$$\langle \bm{p},\Gamma\rangle:=\inf\{\langle  \bm{p},\bm{q}\rangle\mid \bm{q}\in  \Gamma\}.$$
Therefore, $a(E_{\bm{p}};\mathbb A^2_k,\mathfrak a)=\langle \bm{p},\bm{1}\rangle-\langle \bm{p},\Gamma\rangle$.

\begin{lem}\label{lem3.3}
With the above notations, there exists $\bm{p}=(p_1,p_2)\in \mathbb N^2$ with $p_1,p_2$ coprime such that $E_{\bm{p}}$ computes {\rm{mld}}$(0;\mathbb A_k^2,\mathfrak{a})$. That is to say, if {\rm{mld}}$(0;\mathbb A_k^2,\mathfrak{a})\geq 0$, there exists $\bm{p}\in \mathbb N^2$ such that
$$\langle \bm{p},\bm{1}\rangle-\langle \bm{p},\Gamma\rangle=\inf\{\langle \bm{q},\bm{1}\rangle-\langle \bm{q},\Gamma\rangle\mid \bm{q}\in \mathbb N^2\}={\rm{mld}}(0;\mathbb A_k^2,\mathfrak{a})$$
and if {\rm{mld}}$(0;\mathbb A_k^2,\mathfrak{a})=-\infty$, there exists $\bm{p}\in \mathbb N^2$ such that
$$\langle \bm{p},\bm{1}\rangle-\langle \bm{p},\Gamma\rangle<0.$$
\end{lem}
\begin{proof}
It suffices to show that there exists a toric prime divisor over $\mathbb A_k^2$ computing {\rm{mld}}$(0;\mathbb A_k^2,\mathfrak{a})$. Take a toric log resolution $f:Y\rightarrow \mathbb A_k^2 $ of the pair $(\mathbb A_k^2, \mathfrak{a\cdot m}_0)$, where $\mathfrak{m}_0$ is the maximal ideal of the origin.  
Define
$$B:=\sum_F (1-a(F;\mathbb A_k^2,\mathfrak a))F,$$
where $F$ runs all prime divisors on $Y$. Then $B$ is a simple normal crossing toric divisor on $Y$ and $a(E;\mathbb A_k^2,\mathfrak{a})=a(E;Y,B)$ for any prime divisor $E$ over $\mathbb A_k^2$.

If {\rm{mld}}$(0;\mathbb A_k^2,\mathfrak{a})\geq 0$, then coefficients in $B$ are no more than 1 and hence
$${\rm{mld}}(0;\mathbb A_k^2,\mathfrak{a})=\min\{a(E;Y,B) \mid \text{$E$ is a prime divisor on $Y$ with center$_X E=0$}\}.$$
So there is an exceptional prime divisor $E$ on $Y$ such that $a(E;Y,B)={\rm{mld}}(0;\mathbb A_k^2,\mathfrak{a})$ and hence $E$ computes ${\rm{mld}}(0;\mathbb A_k^2,\mathfrak{a})$. 

If {\rm{mld}}$(0;\mathbb A_k^2,\mathfrak{a})=-\infty$, there exists a component of $B$ with coefficient $>1$. Then there exists a toric prime divisor $E$ over $Y$ with $\text{center}_X E=0$ such that $a(E;Y,B)<0$ and hence $E$ computes ${\rm{mld}}(0;\mathbb A_k^2,\mathfrak{a})$. 
\end{proof}
\begin{lem}\label{lem3.4}
With the above notations, the followings are equivalent:

{\rm(1)} {\rm{mld}}$(0;\mathbb A^2_k,\mathfrak{a})\geq 0$;

{\rm(2)} $\bm{1}\in \Gamma$.
\end{lem}
\begin{proof} 
If $\bm{1}\in \Gamma$, by Lemma \ref{lem3.3} there is $\bm{p}=(p_1,p_2)\in \mathbb N^2$ with $p_1,p_2$ coprime such that $E_{\bm{p}}$ computes {\rm{mld}}$(0;\mathbb A_k^2,\mathfrak{a})$. Since $\bm{1}\in \Gamma$, we have $\langle \bm{p},\bm{1}\rangle-\langle \bm{p},\Gamma\rangle\geq 0$, which implies that {\rm{mld}}$(0;\mathbb A_k^2,\mathfrak{a})\geq 0$.

If $\bm{1}\notin \Gamma$, there exists a  1-dimensional face of $\Gamma$ with normal vector $\textbf{e}=(e_1,e_2)$ such that $e_1+e_2< \langle \textbf{e},\Gamma \rangle$. 
After some small perturbations, we may suppose that $e_1,e_2$ are positive rational numbers and the above inequality still holds. Multiplying with some positive rational number, we may suppose that $e_1,e_2$ are coprime positive integers. Since $\langle \bm{e},\bm{1}\rangle<\langle \bm{e},\Gamma\rangle$, we have $a(E_{\bm{e}};\mathbb A_k^2,\mathfrak a)<0$, which implies {\rm{mld}}$(0;\mathbb A_k^2,\mathfrak{a})=-\infty$.
\end{proof}

\section{Proof of the main theorem}
\begin{lem}\label{lem4.1}
Let $\lambda$ be a positive real number.
If $a,b\in \lambda \mathbb Z$ satisfy $1<a\leq b\leq 2$, then 
$$\left\lfloor \frac{a}{a-1}\right\rfloor_{\lambda} +a\geq \left\lfloor \frac{b}{b-1}\right\rfloor_{\lambda} +b.$$
\end{lem}
\begin{proof}
Write $a=n\lambda$ and $b=(n+m)\lambda$ where $n,m\in \mathbb Z$. Since $1<a\leq b\leq 2$, we have $1<n\lambda\leq (n+m)\lambda\leq 2.$ So we have
$$\frac{n\lambda}{n\lambda-1}-\frac{(n+m)\lambda}{(n+m)\lambda-1}=\frac{m\lambda}{(n\lambda-1)((n+1)\lambda-1)}\geq m\lambda.$$
Hence $$\left\lfloor\frac{n\lambda}{n\lambda-1}\right\rfloor_{\lambda}\geq \left\lfloor\frac{(n+m)\lambda}{(n+m)\lambda-1}\right\rfloor_{\lambda}+m\lambda,$$
which implies that
$$\left\lfloor\frac{n\lambda}{n\lambda-1}\right\rfloor_{\lambda}+n\lambda\geq \left\lfloor\frac{(n+m)\lambda}{(n+m)\lambda-1}\right\rfloor_{\lambda}+(n+m)\lambda.$$
\end{proof}

\begin{lem}\label{lem4.2}
Let $\bm{a},\bm{b}\in \mathbb R_{\geq 0}^2$ such that

{\rm(1)} there is $\gamma \in \mathbb R_{>0}$  such that $1+\gamma\leq \bm{a}_y\leq 2$,

{\rm(2)} $\bm{a}_x<\bm{b}_x$ and $\bm{a}_y>\bm{b}_y$,

{\rm(3)} $\bm{1} \in \overline{\bm{ab}}\cup \overline{\bm{ab}}^{+}$,

{\rm(4)} there is $\lambda\in \mathbb R_{>0}$ such that $\bm{a}-\bm{b}\in (\lambda\mathbb Z)^2$.\\
Then $$\frac{\bm{a}_y-\bm{b}_y+\bm{b}_x-\bm{a}_x}{\lambda}\leq  \left\lfloor \frac{\gamma+1}{\lambda \gamma}\right\rfloor+\left\lceil \frac{\gamma+1}{\lambda}\right\rceil.$$
\end{lem}
\begin{proof}
Since $\bm{1} \in \overline{\bm{ab}}\cup \overline{\bm{ab}}^{+}$, we have 
$$\bm{b}_x-\bm{a}_x\leq \frac{(1-\bm{a}_x)(\bm{a}_y-\bm{b}_y)}{\bm{a}_y-1}.$$
Note that $\bm{b}_x-\bm{a}_x\in \lambda \mathbb Z$ and $\bm{a}_y\geq \gamma+1$, we have 
\begin{equation}\label{key}
\begin{aligned}
\bm{b}_x-\bm{a}_x&\leq \left\lfloor \frac{(1-\bm{a}_x)(\bm{a}_y-\bm{b}_y)}{\bm{a}_y-1} \right\rfloor_{\lambda}\\
&\leq \left\lfloor \frac{\bm{a}_y}{\bm{a}_y-1} \right\rfloor_{\lambda}\\
&\leq \left\lfloor \frac{\gamma+1}{\gamma} \right\rfloor _{\lambda}.
\end{aligned}
\end{equation}

If $\bm{a}_y-\bm{b}_y\leq \gamma+1$, then 
\begin{align*}
\frac{\bm{a}_y-\bm{b}_y+\bm{b}_x-\bm{a}_x}{\lambda} &\leq \frac{1}{\lambda}\left(\gamma+1+\left\lfloor \frac{\gamma+1}{\gamma} \right\rfloor _{\lambda}\right)\\
& \leq \left\lceil \frac{\gamma+1}{\lambda}\right\rceil+\left\lfloor \frac{\gamma+1}{\lambda\gamma} \right\rfloor.
\end{align*}

If $\bm{a}_y-\bm{b}_y>\gamma+1$, let $l=\lceil \gamma+1\rceil_{\lambda}$, then $l\leq \bm{a}_y-\bm{b}_y$ since $\bm{a}_y-\bm{b}_y\in \lambda\mathbb Z$. Apply Lemma \ref{lem4.1} to conclude that 
\begin{align}\label{eq421}
\left\lfloor\frac{\bm{a}_y-\bm{b}_y}{\bm{a}_y-\bm{b}_y-1} \right\rfloor_{\lambda}+\bm{a}_y-\bm{b}_y\leq \left\lfloor \frac{l}{l-1}\right\rfloor_{\lambda} +l.
\end{align}
It follows from the first inequality in (\ref{key}) that
\begin{align}\label{eq422}
\bm{b}_x-\bm{a}_x\leq \left\lfloor \frac{\bm{a}_y-\bm{b}_y}{\bm{a}_y-\bm{b}_y-1} \right\rfloor_{\lambda}.
\end{align}
Finally, (\ref{eq421}), (\ref{eq422}) and the fact that $l=\lceil \gamma+1\rceil_{\lambda}\geq \gamma+1$ imply
\begin{align*}
\frac{1}{\lambda}(\bm{a}_y-\bm{b}_y+\bm{b}_x-\bm{a}_x) 
&\leq \frac{1}{\lambda}\left(\left\lfloor\frac{\bm{a}_y-\bm{b}_y}{\bm{a}_y-\bm{b}_y-1} \right\rfloor_{\lambda}+\bm{a}_y-\bm{b}_y\right)\\
&\leq \frac{1}{\lambda}\left(\left\lfloor \frac{l}{l-1}\right\rfloor_{\lambda} +l\right)\\
&\leq \frac{1}{\lambda}\left(\left\lfloor \frac{\gamma+1}{\gamma}\right\rfloor_{\lambda} +\lceil \gamma+1\rceil_{\lambda}\right)\\
&= \left\lfloor \frac{\gamma+1}{\lambda\gamma}\right\rfloor +\left\lceil \frac{\gamma+1}{\lambda}\right\rceil.
\end{align*}
\end{proof}

\begin{lem}\label{lem4.3}
Let $\bm{a},\bm{b},\bm{c}\in \mathbb R_{\geq 0}^2$. Suppose $\bm{a}_y>1$ and $\bm{a}\in \bigtriangleup$, where 
$$\bigtriangleup:=\{(x,y)\in\mathbb{R}_{\geq 0}^2 \mid y\leq 2-x\}.$$
Then
\begin{itemize}
\item[(1)] if $\bm{b}\notin \bigtriangleup$, $\bm{b}_x>\bm{a}_x$ and $1\leq \bm{b}_y < \bm{a}_y$, then $\bm{1}\in \overline{\bm{ab}}^-$;
\item[(2)] if $\bm{c}\notin \bigtriangleup$, $\bm{c}_x<\bm{a}_x$ and $\bm{c}_y > \bm{a}_y$, then $\bm{1}\in \overline{\bm{ca}}^+$.
\end{itemize}
\end{lem}
\begin{proof}
It is easy to check by plotting the graph.
\end{proof}

\begin{proof}[Proof of Theorem \ref{thm1.3} (1)]
By the argument in Section 2, we may suppose that $A=\mathbb A_k^2$ and $\mathfrak a=\prod_{j=1}^r \mathfrak{a}_j^{\lambda_j}$ is a monomial $\mathbb R$-ideal on $A$. For short we denote the Newton polytope $\Gamma(\mathfrak a)$ by $\Gamma$. Since {\rm{mld}}$(0;A,\mathfrak{a})\geq 0$, by Lemma \ref{lem3.4} we have $\bm{1}\in \Gamma$, which implies that no vertices of $\Gamma$ locate in $(1,+\infty)\times(1,+\infty)$.  Let $\bm{a}_1,\cdots,\bm{a}_{n+m+t}$ be vertices of $\Gamma$ such that  $\bm{a}_1,\cdots,\bm{a}_{n}\in [0,1]\times(1,+\infty)$, $\bm{a}_{n+1},\cdots,\bm{a}_{n+m}\in [0,1]\times[0,1]$, $\bm{a}_{n+m+1},\cdots, \bm{a}_{n+m+t} \in (1,+\infty)\times [0,1]$ and $(\bm{a}_1)_y>\cdots>(\bm{a}_{n+m+t})_y$. 
Then 
$$\langle \bm{q},\bm{1}\rangle-\langle \bm{q},\Gamma\rangle=\max_{1\leq i\leq n+m+t}\{\langle \bm{q},\bm{1}-\bm{a}_i\rangle\}\quad \text{ for any } \bm{q}\in \mathbb N^2.$$
We denote
$$\bm{a}_0:=((\bm{a}_1)_x,+\infty) \quad \text{and}\quad 
\bm{a}_{n+m+t+1}:=(+\infty,(\bm{a}_{n+m+t})_y).$$
Define $$\bm{b}_i:=\Big((\bm{a}_i)_y-(\bm{a}_{i+1})_y,(\bm{a}_{i+1})_x-(\bm{a}_i)_x\Big)$$ for $i=0,\cdots,n+m+t$, then
$$(\bm{b}_{n+m+t})_y/(\bm{b}_{n+m+t})_x>\cdots>(\bm{b}_0)_y/(\bm{b}_0)_x.$$ 
Note that $(\bm{b}_{n+m+t})_y/(\bm{b}_{n+m+t})_x=+\infty$ and $(\bm{b}_{0})_y/(\bm{b}_{0})_x=0$. It is easy to see that if $\bm{q}\in \mathbb N^2$ satisfies 
$$(\bm{b}_{i-1})_y/(\bm{b}_{i-1})_x\leq\bm{q}_y/\bm{q}_x\leq(\bm{b}_i)_y/(\bm{b}_i)_x$$
for some $i=1,\cdots,n+m+t$, then $\langle \bm{q},\bm{1}\rangle-\langle \bm{q},\Gamma\rangle=\langle \bm{q},\bm{1}-\bm{a}_i\rangle$.

There are following two cases:

(1)  $(\bm{b}_n)_y/(\bm{b}_n)_x < 1 <  (\bm{b}_{n+m})_y/(\bm{b}_{n+m})_x$. Then $m>0$ and there exists $i_0 \in \{n+1, \cdots,n+m\} $ such that $\langle\bm{1},\bm{1}\rangle-\langle\bm{1},\Gamma\rangle=\langle\bm{1},\bm{1}-\bm{a}_{i_0}\rangle$. For any $\bm{q}\in \mathbb{N}^2$, since $\bm{1}-\bm{a}_{i_0}\in [0,1]\times[0,1]$, we have 
$$\langle\bm{q},\bm{1}\rangle-\langle\bm{q},\Gamma\rangle\geq \langle\bm{q},\bm{1}-\bm{a}_{i_0}\rangle \geq \langle\bm{1},\bm{1}-\bm{a}_{i_0}\rangle=\langle\bm{1},\bm{1}\rangle-\langle\bm{1},\Gamma\rangle.$$ 
Therefore $E_{\bm{1}}$ with log discrepancy 2 computes $\text{mld}(0,A,\mathfrak a)$ and the proof is completed.

(2)  $(\bm{b}_n)_y/(\bm{b}_n)_x\geq 1$ or $(\bm{b}_{n+m})_y/(\bm{b}_{n+m})_x\leq 1$. We may suppose the former holds (if not we can replace $\Gamma$ by its reflection along the diagonal). Then $n>0$ (since $(\bm{b}_0)_y/(\bm{b}_0)_x$=0) and
\begin{align}\label{eq131}
(\bm{a}_{n+1})_x-(\bm{a}_{n})_x\geq (\bm{a}_{n})_y-(\bm{a}_{n+1})_y.
\end{align}
It follows from $\bm{1}\in \Gamma$ that 
\begin{align}\label{eq132}
\bm{1}\in\overline{\bm{a}_n\bm{a}_{n+1}}^{+}\cup \overline{\bm{a}_n\bm{a}_{n+1}}.
\end{align}
Hence $\bm{a}_{n+1}\neq (+\infty,(\bm{a}_n)_y)$ since $(\bm{a}_n)_y>1$. 
So $m+t>0$ and we do not need to worry about that $\bm{a}_n$ or $\bm{a}_{n+1}$ is an infinite point.

Apply Lemma \ref{lem3.2} to obtain that $(\bm{a}_n)_y>1+\gamma$. It follows from (\ref{eq131}) and (\ref{eq132}) that $(\bm{a}_n)_y\leq 2$.  By Lemma \ref{lem3.1}, there exists $j=1,\cdots,r$ and $0<\alpha\leq 1$ such that $\alpha\bm{b}_{n}/\lambda_j \in \mathbb N^2$. Denote $\alpha\bm{b}_{n}/\lambda_j$ by $\bm{b}'$. We apply Lemma \ref{lem4.2} to conclude that 
\begin{align*}
 \bm{b}'_x+\bm{b}'_y &\leq \left\lfloor \frac{\gamma+1}{\lambda_j \gamma}\right\rfloor+\left\lceil \frac{\gamma+1}{\lambda_j}\right\rceil\\
 & \leq \left\lfloor \frac{\gamma+1}{e \gamma}\right\rfloor+\left\lceil \frac{\gamma+1}{e}\right\rceil,
\end{align*}
where the second inequality follows from $\lambda_j \geq e$.
Since $\bm{b}'_y/\bm{b}'_x=(\bm{b}_n)_y/(\bm{b}_n)_x$,
\begin{align}\label{last}
\langle\bm{b}',\bm{1}\rangle-\langle\bm{b}',\Gamma\rangle=\langle\bm{b}',\bm{1}-\bm{a}_n\rangle=\langle\bm{b}',\bm{1}-\bm{a}_{n+1}\rangle.
\end{align}

Let $\bm{p}\in \mathbb{N}^2$ such that $E_{\bm{p}}$ computes $\text{mld}(0,\mathbb A,\mathfrak a)$, i.e.,
\begin{align}\label{eq133}
\langle\bm{p},\bm{1}\rangle-\langle\bm{p},\Gamma\rangle=\inf\{\langle\bm{q},\bm{1}\rangle-\langle\bm{q},\Gamma\rangle\mid \bm{q}\in \mathbb N^2\}.
\end{align}
We may suppose that
\begin{align}
\langle\bm{p},\bm{1}\rangle-\langle\bm{p},\Gamma\rangle<\langle\bm{b}',\bm{1}\rangle-\langle\bm{b}',\Gamma\rangle.
\end{align}
Indeed, if not, then $E_{\bm{b}'}$ computes the mld with its log discrepancy satisfying the inequality and hence the proof is completed.

As $k_{E_{\bm{p}}}+1=\bm{p}_x+\bm{p}_y$, it is enough to show that $\bm{p}_x\leq \bm{b}'_x$ and $\bm{p}_y\leq \bm{b}'_y$. 
There are following four subcases:

(2a) $\bm{p}_y/\bm{p}_x\leq \bm{b}'_y/\bm{b}'_x$. Since
\begin{align*}
\langle\bm{p},\bm{1}-\bm{a}_n\rangle\leq \langle\bm{p},\bm{1}\rangle-\langle\bm{p},\Gamma\rangle< \langle\bm{b}',\bm{1}\rangle-\langle\bm{b}',\Gamma\rangle=\langle\bm{b}',\bm{1}-\bm{a}_n\rangle,
\end{align*}
we have
\begin{align}\label{1}
(\bm{p}_x-\bm{b}'_x)(1-(\bm{a}_n)_x)<(\bm{p}_y-\bm{b}'_y)((\bm{a}_n)_y-1).
\end{align}
Note that $1-(\bm{a}_n)_x \geq 0$ and $(\bm{a}_n)_y-1>0$. We claim that $\bm{p}_x\leq \bm{b}'_x$. Indeed, if this is not the case, then (\ref{1}) implies that
$$\frac{1-(\bm{a}_n)_x}{(\bm{a}_n)_y-1}<\frac{\bm{p}_y-\bm{b}'_y}{\bm{p}_x-\bm{b}'_x}\leq \frac{\bm{b}'_y}{\bm{b}'_x},$$
where the last inequality comes from  $\bm{p}_y/\bm{p}_x\leq \bm{b}'_y/\bm{b}'_x$.
However, (\ref{eq132}) implies that
$$\frac{1-(\bm{a}_n)_x}{(\bm{a}_n)_y-1}\geq \frac{(\bm{a}_{n+1})_x-(\bm{a}_n)_x}{(\bm{a}_n)_y-(\bm{a}_{n+1})_y} = \frac{\bm{b}'_y}{\bm{b}'_x},$$
which leads to a contradiction. Therefore, $\bm{p}_x\leq \bm{b}'_x$. Then $\bm{p}_y\leq \bm{b}'_y$ since $\bm{p}_y/\bm{p}_x\leq \bm{b}'_y/\bm{b}'_x$.

(2b) $\bm{p}_y/\bm{p}_x> \bm{b}'_y/\bm{b}'_x$ and $\bm{p}_y\leq \bm{b}'_y$. Then $\bm{p}_x\leq \bm{b}'_x$.

(2c) $\bm{p}_y/\bm{p}_x> \bm{b}'_y/\bm{b}'_x$, $\bm{p}_y> \bm{b}'_y$ and $\bm{p}_x \leq \bm{b}'_x$. Let $\bm{p}'=(\bm{p}_x,\bm{b}'_y)$, then $\bm{p}'_y/\bm{p}'_x\geq \bm{b}'_y/\bm{b}'_x=(\bm{b}_n)_y/(\bm{b}_n)_x$. Therefore there exists $j_0 \in \{n+1, \cdots,n+m+t\} $ such that 
\begin{align}\label{bing1}
\langle\bm{p}',\bm{1}\rangle-\langle\bm{p}',\Gamma\rangle=\langle\bm{p}',\bm{1}-\bm{a}_{j_0}\rangle.
\end{align}
Since $1-(\bm{a}_{j_0})_y\geq 0$, we have
\begin{align}\label{bing2}
\langle\bm{p}',\bm{1}-\bm{a}_{j_0}\rangle\leq \langle\bm{p},\bm{1}-\bm{a}_{j_0}\rangle \leq \langle\bm{p},\bm{1}\rangle-\langle\bm{p},\Gamma\rangle\leq \langle\bm{p}',\bm{1}\rangle-\langle\bm{p}',\Gamma\rangle,
\end{align}
where the last inequality comes from (\ref{eq133}). It follows from \eqref{bing1} and \eqref{bing2} that 
$\langle\bm{p},\bm{1}\rangle-\langle\bm{p},\Gamma\rangle= \langle\bm{p}',\bm{1}\rangle-\langle\bm{p}',\Gamma\rangle$. We therefore obtain a prime divisor $E_{\bm{p}'}$ computing the minimal log discrepancy with its log discrepancy 
$\bm{p}'_x+\bm{p}'_y\leq \bm{b}'_x+\bm{b}'_y$.

(2d) $\bm{p}_y/\bm{p}_x> \bm{b}'_y/\bm{b}'_x$, $\bm{p}_y> \bm{b}'_y$ and $\bm{p}_x > \bm{b}'_x$. Since
$$\langle\bm{p},\bm{1}-\bm{a}_{n+1}\rangle\leq \langle\bm{p},\bm{1}\rangle-\langle\bm{p},\Gamma\rangle< \langle\bm{b}',\bm{1}\rangle-\langle\bm{b}',\Gamma\rangle=\langle\bm{b}',\bm{1}-\bm{a}_{n+1}\rangle,$$
we have
\begin{align}\label{2}
(\bm{p}_x-\bm{b}'_x)((\bm{a}_{n+1})_x-1)>(\bm{p}_y-\bm{b}'_y)(1-(\bm{a}_{n+1})_y).
\end{align}

Note that $1-(\bm{a}_{n+1})_y\geq 0$. If $1-(\bm{a}_{n+1})_y= 0$, then (\ref{eq132}) implies that $(\bm{a}_{n+1})_x\leq 1$. On the other hand, $\bm{p}_x-\bm{b}'_x>0$. This contradicts (\ref{2}).

If $1-(\bm{a}_{n+1})_y>0$, then (\ref{2}) implies that
$$\frac{(\bm{a}_{n+1})_x-1}{1-(\bm{a}_{n+1})_y}> \frac{\bm{p}_y-\bm{b}'_y}{\bm{p}_x-\bm{b}'_x}> \frac{\bm{b}'_y}{\bm{b}'_x},$$
where the last inequality comes from  $\bm{p}_y/\bm{p}_x>\bm{b}'_y/\bm{b}'_x$.
However, (\ref{eq132}) implies that
$$\frac{(\bm{a}_{n+1})_x-1}{1-(\bm{a}_{n+1})_y}\leq \frac{(\bm{a}_{n+1})_x-(\bm{a}_n)_x}{(\bm{a}_n)_y-(\bm{a}_{n+1})_y} = \frac{\bm{b}'_y}{\bm{b}'_x},$$
which leads to a contradiction. 
\end{proof}

\begin{proof}[Proof of Theorem \ref{thm1.3} (2)]
By the argument in Section 2, we may suppose that $A=\mathbb A_k^2$ and $\mathfrak a=\prod_{j=1}^r \mathfrak{a}_j^{\lambda_j}$ is a monomial $\mathbb R$-ideal on $A$. For short we denote the Newton polytope $\Gamma(\mathfrak a)$ by $\Gamma$. Since {\rm{mld}}$(0;A,\mathfrak{a})=-\infty$, by Lemma \ref{lem3.4} we have $\bm{1}\notin \Gamma$. Let $\bm{a}_1,\cdots,\bm{a}_{n}$ be vertices of $\Gamma$ such that $(\bm{a}_1)_y>\cdots>(\bm{a}_{n})_y$. Denote $\bm{a}_0=((\bm{a}_1)_x,+\infty)$ and  $\bm{a}_{n+1}=(+\infty,(\bm{a}_{n})_y)$.

We may suppose that $\Gamma$ is convenient, i.e $\Gamma$ meets both $x$-axis and $y$-axis. Indeed, if this is not the case, replacing each $\mathfrak a_i$ by the ideal generated by $\mathfrak a_i, x^{m}$ and $y^{m}$ for a sufficiently large integer $m$, we obtain a monomial $\mathbb R$-ideal $\widetilde{\mathfrak{a}}$ with  its Newton polytope convenient and containing the original one. Every divisor $E$ with $a(E;A,\widetilde{\mathfrak{a}})<0$ also satisfies $a(E;A,\mathfrak{a})<0$, so we may replace the original monomial $\mathbb R$-ideal $\mathfrak{a}$ by the new one $\widetilde{\mathfrak{a}}$. Hence we may suppose that $\Gamma$ is convenient.

If $\langle\bm{1},\Gamma\rangle>2$, then $\langle\bm{1},\bm{1}\rangle-\langle\bm{1},\Gamma\rangle<0$, thus $E_{\bm{1}}$ with log discrepancy 2 computes ${\rm{mld}}(0;A,\mathfrak{a})$ and the proof is completed.

If $\langle\bm{1},\Gamma\rangle\leq 2$, there exists $i_0=1,\cdots n$ such that $\bm{a}_{i_0}\in \bigtriangleup$, where 
$$\bigtriangleup=\{(x,y)\in\mathbb{R}_{\geq 0}^2 \mid y\leq 2-x\}.$$ Since $\bm{1}\notin \Gamma$, we have $\bm{a}_{i_0}\notin [0,1]\times[0,1]$. So we may assume that
$(\bm{a}_{i_0})_y>1$
(if not we can replace $\Gamma$ by its reflection along the diagonal). Let
$$j_0:=\max\{i=1,\cdots n \mid \bm{a}_{i}\in \bigtriangleup\},$$
then $j_0\geq i_0$. We claim that $(\bm{a}_{j_0})_y>1$. Indeed, if this is not the case, then
$$\bm{1}\in \{(1-\alpha)\bm{a}_{i_0}+\alpha \bm{a}_{j_0}+\beta \bm{q}\mid 0\leq \alpha \leq 1, \beta\geq 0, \bm{q}=(1,0) \},$$
which follows from $(\bm{a}_{i_0})_y>1, (\bm{a}_{j_0})_y\leq 1$ and $\bm{a}_{i_0},\bm{a}_{j_0}\in \bigtriangleup$. 
This contradicts the fact $\bm{1}\notin \Gamma$ and hence the claim $(\bm{a}_{j_0})_y>1$ holds. Since  $(\bm{a}_{n})_y=0$ (because $\Gamma$ is convenient), we have $j_0<n$. 

By the definition of $j_0$, we have $\bm{a}_{j_0+1}\notin \bigtriangleup$. We claim that $\bm{1}\in \overline{\bm{a}_{j_0}\bm{a}_{j_0+1}}^-$. In fact, if not, by Lemma \ref{lem4.3} (1) we have $(\bm{a}_{j_0+1})_y<1$. Then
$$\bm{1}\in \{(1-\alpha)\bm{a}_{j_0}+\alpha \bm{a}_{j_0+1}+\beta \bm{q}\mid 0\leq \alpha \leq 1, \beta\geq 0, \bm{q}=(1,0) \},$$ which follows from $(\bm{a}_{j_0})_y>1, (\bm{a}_{j_0+1})_y<1$ and $\bm{1}\in \overline{\bm{a}_{j_0}\bm{a}_{j_0+1}}^+\cup \overline{\bm{a}_{j_0}\bm{a}_{j_0+1}}$. This contradicts the fact $\bm{1}\notin \Gamma$ and hence the claim $\bm{1}\in \overline{\bm{a}_{j_0}\bm{a}_{j_0+1}}^-$ holds. 

Let $$l_0:=\min\{i=1,\cdots,n-1 \mid \bm{1}\in \overline{\bm{a}_i\bm{a}_{i+1}}^{-}\}.$$
Then $l_0\leq j_0$. We claim that $\bm{a}_{l_0}\in \bigtriangleup$. In fact, if not, then $l_0<j_0$ and by Lemma \ref{lem4.3} (2) we have $\bm{1}\in \overline{\bm{a}_{l_0}\bm{a}_{j_0}}^+$, which yields $\bm{1}\in \overline{\bm{a}_{l_0}\bm{a}_{l_0+1}}^+$ and leads to a contradiction with the definition of $l_0$.

For short we denote $\bm{a}_{l_0}$ (resp. $\bm{a}_{l_0+1}$) by $\bm{a}$ (resp. $\bm{b}$). Then $\bm{a}\in \bigtriangleup$, $\bm{1}\in \overline{\bm{ab}}^-$ and $ \bm{a}_y=(\bm{a}_{l_0})_y\geq (\bm{a}_{j_0})_y>1$. Apply Lemma \ref{lem3.2} to obtain that $\bm{a}_y\geq 1+\gamma$. By Lemma \ref{lem3.1}, there exists $0<\alpha\leq 1$ and $j=1,\cdots,r$ such that $(\bm{a}-\bm{c})\in (\lambda_j\mathbb Z)^2$ where $\bm{c}=(1-\alpha) \bm{a}+\alpha \bm{b}$. For short we denote $\lambda_j$ by $\lambda$, then $\lambda \geq e$. 
Let $$t:=\max\{t'\in \mathbb Z\mid \bm{1}\in \overline{\bm{ac'}}^{-}, \text{ where } \bm{c'}=(\bm{c}_x-n'\lambda,\bm{c}_y)\}$$
and let $\bm{d}=(\bm{c}_x-t\lambda,\bm{c}_y).$ 
Then $\bm{1}\in \overline{\bm{ad}}^-$ and $\bm{b}\in \overline{\bm{ad}}\cup \overline{\bm{ad}}^+$. 

Let $\bm{f}=\bm{a}_{l_0-1}$ (recall that $\bm{a}=\bm{a}_{l_0}$). If $l_0=1$, then $\bm{f}=(\bm{a}_x,+\infty)$, hence $\bm{f}\in \overline{\bm{ad}}^+$. If $l_0>1$, by the definition of $l_0$, we have $\bm{1}\notin \overline{\bm{fa}}^-$ while  $\bm{1}\in \overline{\bm{ad}}^-$. It follows that $\bm{f}\in \overline{\bm{ad}}^+$. 

Since $\bm{f}=\bm{a}_{l_0-1}\in \overline{\bm{ad}}^+, \bm{b}=\bm{a}_{l_0+1} \in \overline{\bm{ad}}\cup \overline{\bm{ad}}^+$ and $\bm{a}=\bm{a}_{l_0}$, we have $\Gamma \subseteq \overline{\bm{ad}}\cup\overline{\bm{ad}}^+$. On the other hand, $\bm{1}\in \overline{\bm{ad}}^-$. This implies that $\langle\bm{p},\bm{1}\rangle-\langle\bm{p},\Gamma\rangle<0$ where 
$$\bm{p}=\frac{(\bm{a}_y-\bm{d}_y,\bm{d}_x-\bm{a}_x)}{\lambda}\in \mathbb N^2.$$
Let $\widetilde{\bm{p}}=\bm{p}/\gcd(\bm{p}_x,\bm{p}_y)$. Then $E_{\widetilde{\bm{p}}}$ computes ${\rm{mld}}(0;A,\mathfrak{a})$ with log discrepancy $\leq\bm{p}_x+\bm{p}_y$.

Let $\bm{d}'=(\bm{d}_x-\lambda,\bm{d}_y)$. If $\bm{d}'_x \leq \bm{a}_x$, then $\bm{d}_x-\bm{a}_x\leq \lambda$. It follows from $\bm{a}\in \bigtriangleup$ and $\bm{1}\in \overline{\bm{ad}}^-$ that the slope of $\overline{\bm{ad}}>-1$. So $\bm{a}_y-\bm{d}_y< \bm{d}_x-\bm{a}_x$. Therefore, 
$$\bm{p}_x+\bm{p}_y=\frac{\bm{d}_x-\bm{a}_x+\bm{a}_y-\bm{d}_y}{\lambda}\leq 2$$
and 
the proof is completed.

If $\bm{d}'_x > \bm{a}_x$, by the definition of $\bm{d}$, we have $\bm{1}\in \overline{\bm{ad}'}\cup  \overline{\bm{ad}'}^{+}$. Applying Lemma \ref{lem4.2} we obtain
$$\frac{\bm{a}_y-\bm{d}'_y+\bm{d}'_x-\bm{a}_x}{\lambda}\leq \left\lfloor \frac{\gamma+1}{\lambda\gamma}\right\rfloor+ \left\lceil\frac{1+\gamma}{\lambda}\right \rceil,$$
which implies that
\begin{align*}
\bm{p}_x+\bm{p}_y &= \frac{\bm{a}_y-\bm{d}_y+\bm{d}_x-\bm{a}_x}{\lambda}\\
&\leq \left\lfloor \frac{\gamma+1}{\lambda\gamma}\right\rfloor +\left\lceil \frac{\gamma+1}{\lambda}\right\rceil+1\\
&\leq \left\lfloor \frac{\gamma+1}{e\gamma}\right\rfloor +\left\lceil \frac{\gamma+1}{e}\right\rceil+1
\end{align*}
and completes the proof.
\end{proof}

\section{Proofs of examples}
\begin{lem}\label{toric}
Let $\mathfrak a$ be a monomial $\mathbb R$-ideal on $\mathbb A_k^2$ which supports on the origin (see Definition \ref{d1} for the definition of the support of $\mathfrak a$). Suppose that there exists a positive number $\ell$ such that any toric prime divisor that computes  $\rm{mld}(0;\mathbb A^2_k,\mathfrak{a})$ satisfies that its log discrepancy $\geq 
\ell$. Then the same conclusion holds for all prime divisors that compute the mld.
\end{lem}
\begin{proof}
Take a toric log resolution $\pi: Y \rightarrow \mathbb A_k^2$ of the pair $(\mathbb A_k^2, \mathfrak a)$. Define
$$B:=\sum_F (1-a(F;\mathbb A_k^2,\mathfrak a))F,$$
where $F$ runs all prime divisors on $Y$ whose center on $\mathbb A_k^2$ is 0. Then $B$ is a simple normal crossing divisor on $Y$ and 
\begin{align*}
a(E;\mathbb A_k^2,\mathfrak a)=a(E;Y,B)
\end{align*}
for any prime divisor $E$ over $\mathbb A_k^2$.

Let $E$ be a prime divisor over $\mathbb A^2_k$ with the center 0 that computes the mld. If $E$ is a divisor on $Y$, then $E$ is a toric divisor and the conclusion holds. If $E$ is exceptional over $Y$, we claim that the center of $E$ on $Y$ is contained in some prime divisor on $Y$ that computes the mld. Indeed, if this is not the case, there are the following two cases:

(1) $\rm{mld}(0;\mathbb A^2_k,\mathfrak{a})=-\infty$. Then $\text{center}_Y E$ is not contained in any component of $B$ with coefficient $>1$. Then $a(E;Y,B)=a(E;Y,B_{\leq 1})\geq 0$, where the second inequality follows from the fact that $B$ is simple normal crossing (see \textbf{Notation} (7) in Section 1 for the definition of $B_{\leq 1}$). This contradicts that $E$ computes the mld.

(2) $\rm{mld}(0;\mathbb A^2_k,\mathfrak{a})=\epsilon\geq 0$. Then $\text{center}_Y E$ is not contained in any component of $B$ with coefficient $\geq 1-\epsilon$. Then $a(E;Y,B)=a(E;Y,B_{< 1-\epsilon})>2\epsilon$, where the second inequality follows from the fact that $B$ is simple normal crossing (see \textbf{Notation} (7) in Section 1 for the definition of $B_{< 1-\epsilon}$). This contradicts that $E$ computes the mld.

Therefore the center of $E$ on $Y$ is contained in some prime divisor $F$ on $Y$ that computes the mld. Since $F$ is a toric divisor, its log dicrepancy $\geq \ell$, which implies that the log dicrepancy of $E \geq \ell$.
\end{proof}

\begin{lem}\label{lem5.0}
Fix a positive integer $n\geq 2$. Let $\Gamma$ be the Newton polytope of the monomial $\mathbb R$-ideal
$\mathfrak a=(x^{n^2},y^{n-1})^{e}$ on $\mathbb A_k^2$ where $e=1/(n-1)+1/n^2$. Let $\bm{p}$ be the vector $(n-1,n^2)$. Then 
$\langle\bm{p},\bm{1}\rangle-\langle\bm{p},\Gamma\rangle=0$
and  
\begin{align*}
\langle\bm{q},\bm{1}\rangle-\langle\bm{q},\Gamma\rangle\geq 0 \quad \text{for any $\bm{q}\in \mathbb N^2$}.
\end{align*}
Moreover, if the equality holds, then 
$\langle\bm{q},\bm{1}\rangle\geq \langle\bm{p},\bm{1}\rangle=n^2+n-1.$
\end{lem}
\begin{proof}
By direct calculation, we obtain $\langle\bm{p},\bm{1}\rangle-\langle\bm{p},\Gamma\rangle=0$. For any $\bm{q}\in \mathbb N^2$,
\begin{align*}
\langle\bm{q},\Gamma\rangle&=\min\left\{en^2 \bm{q}_x, e(n-1)\bm{q}_y\right\}.
\end{align*}
Since
$$\langle \bm{q},\bm{1}\rangle=\bm{q}_x+\bm{q}_y=\frac{n-1}{n^2+n-1}(en^2 \bm{q}_x)+\frac{n^2}{n^2+n-1}(e(n-1)\bm{q}_y),$$
we have either $$en^2 \bm{q}_x\leq \langle \bm{q},\bm{1}\rangle\leq e(n-1)\bm{q}_y$$ or $$e(n-1)\bm{q}_y\leq \langle \bm{q},\bm{1}\rangle\leq en^2 \bm{q}_x.$$
Therefore, $\langle\bm{q},\bm{1}\rangle\geq \langle\bm{q},\Gamma\rangle$ and the equality holds if and only if 
$$en^2 \bm{q}_x= \langle \bm{q},\bm{1}\rangle= e(n-1)\bm{q}_y.$$
Note that $n^2$ is coprime with $n-1$. Hence in this case there is a positive integer $a$ such that $\bm{q}=a\bm{p}=(a(n-1),an^2)$, which implies that $\langle\bm{q},\bm{1}\rangle\geq \langle\bm{p},\bm{1}\rangle$.
\end{proof}

\begin{proof}[Proof of Example \ref{e1}]
It follows from Lemma \ref{toric} and Lemma \ref{lem5.0}.
\end{proof}

\begin{lem}\label{lem5.1}
Fix a positive integer $n$. Let $\Gamma$ be the Newton polytope of the monomial $\mathbb R$-ideal
$(x^{n^2+n+1},y^{n+1})^{1/n}$ on $\mathbb A_k^2$ and let $\bm{p}$ be the vector $(n+1,n^2+n+1)$. Then 

\rm{(1)} $\langle\bm{p},\bm{1}\rangle-\langle\bm{p},\Gamma\rangle<0$;

\rm{(2)} $\langle\bm{q},\bm{1}\rangle\geq \langle\bm{p},\bm{1}\rangle=(n+1)^2+1$ for any $\bm{q}\in \mathbb N^2$ such that $\langle\bm{q},\bm{1}\rangle-\langle\bm{q},\Gamma\rangle<0$. 

\end{lem}
\begin{proof}
By direct calculation, we obtain $\langle\bm{p},\bm{1}\rangle-\langle\bm{p},\Gamma\rangle=-1/n$.
Let $\bm{q}\in \mathbb N^2$ such that $\langle\bm{q},\bm{1}\rangle-\langle\bm{q},\Gamma\rangle<0$. Then
$$\langle\bm{q},\Gamma\rangle=\min\left\{\frac{n^2+n+1}{n}\bm{q}_x,\frac{n+1}{n}\bm{q}_y\right\}.$$

If $(n+1)\bm{q}_y\leq(n^2+n+1)\bm{q}_x$, then
\begin{align*}
\bm{q}_y\leq \left(n+\frac{1}{n+1}\right)\bm{q}_x
\end{align*} and 
\begin{align}\label{o1}
\langle\bm{q},\bm{1}\rangle-\langle\bm{q},\Gamma\rangle=\bm{q}_x-\frac{1}{n}\bm{q}_y<0.
\end{align}
Hence $$n\bm{q}_x<\bm{q}_y\leq \left(n+\frac{1}{n+1}\right)\bm{q}_x.$$
Note that $\bm{q}_x$ and $\bm{q}_y$ are integers, we have $\bm{q}_x\geq n+1$. Hence (\ref{o1}) implies that $\bm{q}_y\geq n^2+n+1$. Therefore $\bm{q}_x+\bm{q}_y\geq (n+1)^2+1$.

If $(n+1)\bm{q}_y\geq(n^2+n+1)\bm{q}_x$, then
\begin{align}\label{o2}
\bm{q}_y\geq \left(n+\frac{1}{n+1}\right)\bm{q}_x>n\bm{q}_x
\end{align} and 
\begin{align*}
\langle\bm{q},\bm{1}\rangle-\langle\bm{q},\Gamma\rangle=\bm{q}_y-\left(n+\frac{1}{n}\right)\bm{q}_x<0.
\end{align*}
Hence $$ n\bm{q}_x<\bm{q}_y<\left(n+\frac{1}{n}\right)\bm{q}_x.$$
Note that $\bm{q}_x$ and $\bm{q}_y$ are integers, we have $\bm{q}_x\geq n+1$. Hence (\ref{o2}) implies that $\bm{q}_y\geq n^2+n+1$. Therefore $\bm{q}_x+\bm{q}_y\geq (n+1)^2+1$.
\end{proof}

\begin{proof}[Proof of Example \ref{e2}]
It follows from Lemma \ref{toric} and Lemma \ref{lem5.1}.
\end{proof}

\end{document}